\documentclass[10pt,a4paper,reqno]{amsart}
\usepackage{amssymb}\usepackage{amsfonts}\usepackage{amsmath}
\usepackage[cp1250]{inputenc}\usepackage{xcolor}
\usepackage{ifthen}\usepackage{graphicx}
\newtheorem{theorem}{Theorem}[section]
\newtheorem{definition}{Definition}[section]
\newtheorem{corollary}{Corollary}[section]

\newtheorem{lemma}{Lemma}[section]

\begin{document}

\title[Functions defined by symmetric \textit{q}-derivative operator]{Subclass of $k$-uniformly starlike functions defined by
symmetric $q$-\textit{derivative operator}}
\author{S. Kanas$^{1}$, \c{S}. Alt\i nkaya$^{2}$, S.
Yal\c{c}\i n$^{3}$}

\address{$^{1}$ Department of Mathematical Analysis,
Faculty of Mathematics and Natural Sciences, University of Rzeszow,
ul. St. Pigonia 1, 35-310 Rzeszow, Poland} \email{skanas@ur.edu.pl}

\address{$^{2, 3}$ Department of Mathematics, Faculty of Arts
and Science, Uludag University, 16059, Bursa, Turkey}
\email{sahsene@uludag.edu.tr, syalcin@uludag.edu.tr}

\maketitle
\renewcommand{\thefootnote}{}


\footnote{\textsuperscript{1} Corresponding author}

\vspace{-5em}
\begin{abstract}
The theory of $q$-analogs frequently occurs in a number of areas,
including the fractals and dynamical systems. The $q$-derivatives
and $q$-integrals play a prominent role in the study of $q$-deformed
quantum mechanical simple harmonic oscillator. In this paper, we define a symmetric
$q$-derivative operator and study  new family of univalent functions defined by use of that operator.
We establish some new relations between
functions satisfying analytic conditions related to conical
sections.
\end{abstract}

\section{Introduction, Definitions and Notations}\setcounter{equation}{0}

The intrinsic properties of $q$-analogs, including the applications
in the study of quantum groups and $q$-deformed superalgebras, study
of fractals and multi-fractal measures, and in chaotic dynamical
systems are known in the literature. Some integral transforms in the
classical analysis have their $q$-analogues in the theory of
$q$-calculus. This has led various researchers in the field of
$q$-theory for extending all the important results involving the
classical analysis to their $q$-analogs.

For the convenience, we provide some basic definitions and concept
details of $q$-calculus which are used in this paper. Throughout
this paper, we will assume that $q$ satisfies the condition $0<q<1.$
We shall follow the notation and terminology of \cite{Gasper 90}. We first
recall the definitions of fractional \textit{q}-calculus operators
of complex valued function $f$.

\begin{definition}[\cite{Gasper 90}]\label{eq4} Let $q\in (0,1)$ and let $\lambda\in \mathbb{C}$. The
$q$-number, denoted $\left[\lambda\right]_{q}$, we define as
\begin{equation} \left[\lambda\right]_{q}=\dfrac{1-q^{\lambda}}{1-q}.\end{equation}
In the case when $\lambda =n\in \mathbb{N}$ we obtain
$[\lambda]_q=1+q+q^2+\cdots+q^{n-1}$, and when  $q\to 1^-$ then
$[n]_q = n$.  The symmetric $q$-number, denoted $\widetilde{\left[
n\right] }_{q}$ is defined as a number
\begin{equation}
\widetilde{\left[ n\right] }_{q}=\dfrac{q^{n}-q^{-n}}{q-q^{-1}},
\end{equation}
that reduces to $n$, in the case when $q\to 1^-$.
\end{definition}
We note that the symmetric $q$-number do not reduce to the defined above $q$-number, and  frequently occurs in the study
of $q$-\textit{deformed quantum mechanical simple harmonic oscillator}
(see \cite{Bredenharn 84}).

Applying the above $q$-numbers we define $q$-derivative and symmetric
$q$-derivative, below.

\begin{definition} [\cite{Jackson 08}] The $q$-derivative of a function $f$,
defined on a subset of $ \mathbb{C}$, is given by
\begin{equation*}
(D_{q}f)(z)=\left\{\begin{array}{lcl}\dfrac{f(z)-f(qz)}{(1-q)z}&
for&z\neq
0,\\
f'(0)&for&z=0.\end{array}\right.
\end{equation*}%
\end{definition}
We note that $\lim\limits_{q\rightarrow 1^{-}}(D_{q}f)(z)=f'(z)$ if
$f$\ is differentiable at $z$. Additionally, if $f(z) =z+a_2z^2+\cdots$, then
\begin{equation}\label{eq5}
(D_{q}f)(z)=1+\overset{\infty }{\underset{n=2}{\sum }}\left[ n\right]_{q}a_{n}z^{n-1}.
\end{equation}

\begin{definition}[\cite{Brahim 2013}] The symmetric $q$-derivative $\widetilde{D}_{q}f$ of a function $f$  is defined as follows:%
\begin{equation} \label{eq6}
(\widetilde{D}_{q}f)(z)=\left\{\begin{array}{lcl}\dfrac{f(qz)-f(q^{-1}z)}{(q-q^{-1})z}&
for&z\neq 0,\\
f'(0)&for& z=0.\end{array}\right.
\end{equation}
\end{definition}
From \eqref{eq6}, we deduce that
$\widetilde{D}_{q}z^{n}=\widetilde{\left[ n\right] }_{q}z^{n-1}$,
and a power series of $\widetilde{D}_{q}f$, when $f(z) =z+a_2z^2+\cdots$, is
\begin{equation*}
(\widetilde{D}_{q}f)(z)=1+\overset{\infty }{\underset{n=2}{\sum }}\widetilde{%
\left[ n\right] }_{q}a_{n}z^{n-1}.
\end{equation*}%

It is easy to check that the following properties hold%
\begin{equation*}\begin{array}{rcl}
\widetilde{D}_{q}(f(z)+g(z))&=&(\widetilde{D}_{q}f)(z)+(\widetilde{D}_{q}g)(z),\\
\widetilde{D}_{q}(f(z)g(z)) &=&g(q^{-1}z)(\widetilde{D}_{q}f)(z)+f(qz)(\widetilde{D}_{q}g)(z) \\
&=&g(qz)(\widetilde{D}_{q}f)(z)+f(q^{-1}z)(\widetilde{D}_{q}g)(z),\\
\widetilde{D}_{q}f(z)&=&D_{q^{2}}f(q^{-1}z).\end{array}
\end{equation*}

The defined above fractional $q$-calculus are the important tools used in a study of various families of
analytic functions, and  in the context of univalent functions was first used in a book chapter by Srivastava \cite{Srivastava1989}.
In  contrast to the Leibniz notation, being a ratio of two infinitisemals,  the  notions  of $q$-derivatives are  plain
ratios. Therefore, it appeared soon a generalization of $q$-calculus in many subjects, such as hypergeometric series, complex analysis, and particle physics.
It is also widely applied in an approximation theory, especially on various operators, which includes convergence of operators to functions
in real and complex domain. In the last twenty years $q$-calculus served as a bridge between mathematics and
physics. The field has expanded explosively, due to the fact that applications of basic
hypergeometric series to the diverse subjects of combinatorics, quantum
theory, number theory, statistical mechanics, are constantly being uncovered.
Specially, the theory of univalent functions can be newly described by using the theory of the $q$-calculus. In recent years,
such $q$-calculus operators as the fractional $q$-integral and fractional $q$-derivative operators were used to
construct several subclasses of analytic functions (see, for example
\cite{Kanas 2014}, \cite{Polatoglu 2016}, \cite{Purohit 2011},
\cite{Ucar 2016}).  In the present paper we study the symmetric $q$-operator, and related problems
involving univalent functions.

Let $\mathcal{A}$ denote the class of functions of the form:
\begin{equation}\label{eq1}
f(z)=z+\overset{\infty }{\underset{n=2}{\sum }}a_{n}z^{n},
\end{equation}
which are analytic in the open unit disk $\mathbb{D}=\left\{z\in
\mathbb{C}:\ \ \left\vert z\right\vert <1\right\}$. Also, let
$\mathcal{S}$, $\mathbf{T}$ be the subclasses of $\mathcal{A}$
consisting of functions which are univalent in $\mathbb{D}$, and
with negative coefficients, respectively.
We denote by $\mathcal{ST}(\alpha )$ ($0\leq \alpha <1$) a subset
of $\mathcal{S}$ consisting of all functions starlike of order
$\alpha $, i.e. such that $\Re\left(zf'(z)/f(z)\right) >\alpha \ \ (z\in
\mathbb{D})$. When $\alpha =0$ the class  $\mathcal{ST}(\alpha )$ becomes the
class $\mathcal{ST}$ of functions $f$ \ that maps $\mathbb{D}$ onto
a starlike domain with respect to the origin. By  $k$\mbox{-}$\mathcal{ST}(\alpha)$ we denote the class of $k$-starlike functions of order $\alpha $, $0\leq \alpha <1$,
that is a class of function $f$, which satisfy a condition%
\begin{equation}\label{kst}
\Re\left( \dfrac{zf^{\prime }(z)}{f(z)}\right) >k\left\vert \dfrac{%
zf^{\prime }(z)}{f(z)}-1\right\vert +\alpha \ \ \ \ (k\geq 0),
\end{equation}
for details see \cite{Kanas 2000b} and \cite{Bharati 97}.

We remark here that the class of $k$-starlike functions of
order $\alpha$ is an extension of the relatively more familiar class
of $k$-starlike functions investigated earlier by Kanas et al.
\cite{Kanas 2000b}, \cite{Kanas 2001}, \cite{Kanas 2009},
\cite{Kanas 2014}, see also \cite{Seker 2011}, \cite{Srivastava 2000}. For the case $k=1$ that class was studied by R\o nning
\cite{Ronning 93}, and called there "a parabolic class". We mention
here that the name $k$-uniformly starlike was incorrectly attributed
to the class of $k$-starlike functions defined by \eqref{kst} (for
$\alpha=0$), and related to the class of $k$-uniformly convex
functions by the well known Alexander relation. A class  of
uniformly starlike functions is due to Goodman \cite{Goodman 91} and
was defined by the condition
\begin{equation}\label{kust}
\Re\left( \dfrac{(z-\zeta)f^{\prime }(z)}{f(z)-f(\zeta)}\right)
>0 \ \ \ \ (z, \zeta \in \mathbb{D}),
\end{equation}
and is completely different that the class $k$-stalike functions.

\begin{definition}\label{def01}
Let $0\leq k<\infty $ and $0\leq \alpha <1.$ By $k$\mbox{-}$\widetilde{\mathcal{ST}}_q(\alpha )$  we denote
the class of functions $f\in \mathcal{A}$ satisfying the condition%
\begin{equation}\label{def1}
\Re\left( \dfrac{z(\widetilde{D}_{q}f)(z)}{f(z)}\right)
>k\left\vert \dfrac{z(\widetilde{D}_{q}f)(z)}{f(z)}-1\right\vert
+\alpha \ \ \ \ (z\in \mathbb{D}).
\end{equation}%
We also set  $k$\mbox{-}$\widetilde{\mathcal{ST}}^-_q(\alpha )\ $ =
$\ k$\mbox{-}$\widetilde{\mathcal{ST}}_q(\alpha )\cap \mathbf{T}$.
We note that $\lim\limits_{q\rightarrow 1^{-}}\
k$\mbox{-}$\widetilde{\mathcal{ST}}_q(\alpha )=\
k$\mbox{-}$\mathcal{ST}(\alpha )$.\end{definition}

Let $\mathcal{P}$ be the Carath\`{e}odory class of functions with positive real part
consisting of all functions $p$ analytic in $\mathbb{D}$ satisfying
$p(0)=1$, and $\Re(p(z))>0$.  Making use of a  properties of the Carath\`{e}odory
functions we may rewrite a definition of $k$\mbox{-}$\widetilde{\mathcal{ST}}^-_q(\alpha )$.
Setting  $p(z)=\dfrac{z(\widetilde{D}_{q}f)(z)}{f(z)}$ we may rewrite a condition \eqref{def1} in a form
$\Re p(z) > k|p(z)-1|+\alpha\ (z\in \mathbb{D})$, or $p\prec p_{k,\alpha}$, where
$p_{k,\alpha}$ is a function with a positive real part, that maps the unit disk onto a domain $\Omega_{k,\alpha}$, described by the
inequality $\Re\, w> k|w-1|+\alpha$ (here $\prec$ denotes a symbol of a subordination of the analytic functions). We note that $\Omega_{k,\alpha}$
is a domain bounded by a conic section, symmetric about real axis and contained in a right half plane.
It is also known  that $p_{k,\alpha}$ has the real
and positive coefficients (see \cite{KanSu, Kanas 2005}).  We will use the following notation $p_{k,\alpha} = 1+P_1 z+P_2 z^2 +\cdots$.

It is known, that if $p\in \mathcal{P}$ has a Taylor
series expansion $p\left( z\right)
=1+B_{1}z+B_{2}z^{2}+B_{3}z^{3}+\cdots$, then $|B_n| \leq 2$ for $n\in \mathbb{N}$ \cite{Pommerenke 75}.

More refinement result was obtained by Grenander and Szeg\"{o} \cite{Gr}.

\begin{lemma}\label{gr}
\cite{Gr} If the function $p\in \mathcal{P}$, then%
\begin{eqnarray*}
2B_{2} &=&B_{1}^{2}+x(4-B_{1}^{2}), \\
4B_{3}
&=&B_{1}^{3}+2(4-B_{1}^{2})B_{1}x-B_{1}(4-B_{1}^{2})x^{2}+2(4-B_{1}^{2})(1-%
\left\vert x\right\vert ^{2})z.
\end{eqnarray*}%
for some x, z with $\left\vert x\right\vert \leq 1$ and $\left\vert
z\right\vert \leq 1.$
\end{lemma}

\section{Fundamental properties}\setcounter{equation}{0}

Several new subclasses of the families  of $k$-starlike and
$k$-uniformly convex functions making use of linear operators and
fractional calculus were studied see, for example \cite{Kanas
2000a}, \cite{Srivastava 2000}, and various interesting properties
were obtained. In light of this, it is of interest to consider the behaviour of the
classes $k$\mbox{-}$\widetilde{\mathcal{ST}}_q(\alpha )$ and
$k$\mbox{-}$\widetilde{\mathcal{ST}}^-_q(\alpha )$ defined by symmetric
\textit{q-}derivative operator. We provide necessary and sufficient
coefficient conditions, distortion bounds, and extreme points. In the first
theorems  we provide a necessary and a necessary and sufficient
conditions to be a member of
$k$\mbox{-}$\widetilde{\mathcal{ST}}_q(\alpha )$ and
$k$\mbox{-}$\widetilde{\mathcal{ST}}^-_q(\alpha )$, respectively.

\begin{theorem}\label{th2.1}
Let $0< q < 1$, and  $f\in \mathcal{S}$ be given by \eqref{eq1}. If
the inequality
\begin{equation}\label{eq7}
\overset{\infty }{\underset{n=2}{\sum }}\left[ \widetilde{\left[ n\right] }%
_{q}(k+1)-(k+\alpha )\right] \left\vert a_{n}\right\vert \leq 1-\alpha
\end{equation}
holds true for some $k$ $\left( 0\leq k<\infty \right) $ and $\alpha $ $%
\left( 0\leq \alpha <1\right) ,$ then $f\in
k$\mbox{-}$\widetilde{\mathcal{ST}}_q(\alpha )$.
\end{theorem}
\begin{proof}
By a Definition \ref{def01}, it suffices to prove that
\begin{equation*}
k\left\vert \dfrac{z(\widetilde{D}_{q}f)(z)}{f(z)}-1\right\vert -\Re%
\left( \dfrac{z(\widetilde{D}_{q}f)(z)}{f(z)}-1\right) <1-\alpha .\
\
\end{equation*}%
Observe that%
\begin{eqnarray*}
k\left\vert \dfrac{z(\widetilde{D}_{q}f)(z)}{f(z)}-1\right\vert -\Re%
\left( \dfrac{z(\widetilde{D}_{q}f)(z)}{f(z)}-1\right) &\leq
&(k+1)\left\vert
\dfrac{z(\widetilde{D}_{q}f)(z)}{f(z)}-1\right\vert \\
&=&(k+1)\left\vert \dfrac{\overset{\infty }{\underset{n=2}{\sum
}}\left(
\widetilde{\left[ n\right] }_{q}-1\right) a_{n}z^{n-1}}{1+\overset{\infty }{%
\underset{n=2}{\sum }}a_{n}z^{n-1}}\right\vert \\
&\leq &(k+1)\dfrac{\overset{\infty }{\underset{n=2}{\sum }}\left( \widetilde{%
\left[ n\right] }_{q}-1\right) \left\vert a_{n}\right\vert }{1-\overset{%
\infty }{\underset{n=2}{\sum }}\left\vert a_{n}\right\vert }.
\end{eqnarray*}%
The last expression is bounded by $1-\alpha $, if the inequality
\eqref{eq7} holds.\end{proof}

The inequality \eqref{eq7} gives a tool to obtain some special
members $k$\mbox{-}$\widetilde{\mathcal{ST}}_q(\alpha )$. For
example we have.

\begin{corollary}
Let $0\leq k<\infty,\ 0 < q < 1$, and $0\le \alpha < 1$. If, for
$f(z)=z+a_{n}z^{n}$, the following inequality
\begin{equation*}
\left\vert a_{n}\right\vert \leq \dfrac{1-\alpha }{\widetilde{\left[
n\right] }_{q}(k+1)-(k+\alpha )}\qquad (n\geq 2)
\end{equation*}%
holds, then $f\in k$\mbox{-}$\widetilde{\mathcal{ST}}_q(\alpha )$.
Specially $f(z) = z+\dfrac{(1-\alpha)q}{q^2(k+1)+1-\alpha}z^2\in
k$\mbox{-}$\widetilde{\mathcal{ST}}_q(\alpha )$.
\end{corollary}

\begin{theorem}\label{th2.2}
Let $0\leq k<\infty,\ 0 < q < 1$, and $0\le \alpha <1$. A necessary
and sufficient condition for $f$ of the form $f(z) = z-a_2z^2-\cdots\ (a_n \ge 0)$ to be in
the class $k$\mbox{-}$\widetilde{\mathcal{ST}}^-_q(\alpha )$ is that
\begin{equation}\label{eq8}
\overset{\infty }{\underset{n=2}{\sum }}\left[ \widetilde{\left[ n\right] }%
_{q}(k+1)-(k+\alpha )\right] a_{n}\leq 1-\alpha.
\end{equation}%
The result is sharp, equality holds for the function $f$ given by%
\begin{equation*}
f(z)=z-\dfrac{1-\alpha }{\widetilde{\left[ n\right] }_{q}(k+1)-(k+\alpha )}%
z^{n}.
\end{equation*}
\end{theorem}

\begin{proof}
In view of Theorem \ref{th2.1}, we need only to prove
the necessity. If $f\in k$\mbox{-}$\widetilde{\mathcal{ST}}^-_q(\alpha )$, then  by $|\Re(z)| \leq \left\vert z\right\vert $ for any $z,$
we obtain
\begin{equation}\label{eq9}
\left\vert \dfrac{\overset{\infty }{1-\underset{n=2}{\sum
}}\widetilde{\left[
n\right] }_{q}a_{n}z^{n-1}}{1-\overset{\infty }{\underset{n=2}{\sum }}%
a_{n}z^{n-1}}-\alpha \right\vert \geq k\left\vert \dfrac{\overset{\infty }{%
\underset{n=2}{\sum }}\left( \widetilde{\left[ n\right] }_{q}-1\right)
a_{n}z^{n-1}}{1+\overset{\infty }{\underset{n=2}{\sum }}a_{n}z^{n-1}}%
\right\vert .
\end{equation}%
Choose values of $z$ on the real axis so that
$\widetilde{D}_{q}f(z)$ is real. Upon clearing the dominator of
\eqref{eq9} and letting $z\rightarrow 1^{-}$ through the real
values, we obtain \eqref{eq8}. This completes the proof.
\end{proof}

\begin{theorem}\label{th3.1}
Let $0\leq k<\infty,\ 0 < q < 1$ and $0\le \alpha <1$. Let the
function $f$ defined by $f(z) = z-a_2z^2-\cdots\ (a_n \ge 0)$ be in the class
$k$\mbox{-}$\widetilde{\mathcal{ST}}^-_q(\alpha )$. Then for
$|z|=r<1$ it holds
\begin{equation}
r-\dfrac{q\left( 1-\alpha \right) }{\left( q^{2}+1\right)
(k+1)-q(k+\alpha )} r^{2} \leq \left\vert f(z)\right\vert \leq
r+\dfrac{q\left( 1-\alpha \right) }{\left( q^{2}+1\right)
(k+1)-q(k+\alpha )}r^{2}.\label{eq10}\end{equation}
Equality in \eqref{eq10} holds true for the function $f$ given by%
\begin{equation}
f(z) = z +\dfrac{q\left( 1-\alpha \right) }{\left( q^{2}+1\right)
(k+1)-q(k+\alpha )}z^{2}. \label{eq11}
\end{equation}
\end{theorem}

\begin{proof}
Since $f\in k$\mbox{-}$\widetilde{\mathcal{ST}}^-_q(\alpha )$, then in view of Theorem \ref{th2.2}, we have%
$$\left[ \widetilde{\left[ 2\right] }_{q}(k+1)-(k+\alpha )\right] \overset{%
\infty }{\underset{n=2}{\sum }}  a_{n}\  \leq\ \overset{\infty }{\underset{n=2}{\sum }}\left[ \widetilde{\left[ n\right] }%
_{q}(k+1)-(k+\alpha )\right]\ |a_{n}| \ \leq 1-\alpha,$$
which gives%
\begin{equation}
\overset{\infty }{\underset{n=2}{\sum }}a_{n}\leq \dfrac{1-\alpha }{%
\widetilde{\left[ 2\right] }_{q}(k+1)-(k+\alpha )}.  \label{eq12}
\end{equation}%
Therefore%
$$\left\vert f(z)\right\vert \leq \left\vert z\right\vert +\overset{\infty }{%
\underset{n=2}{\sum }}a_{n}\left\vert z\right\vert ^{n} \leq
r+\dfrac{q\left( 1-\alpha \right) }{\left( q^{2}+1\right)
(k+1)-q(k+\alpha )}r^{2},$$
and
$$\left\vert f(z)\right\vert \geq \left\vert z\right\vert -\overset{\infty }{%
\underset{n=2}{\sum }}a_{n}\left\vert z\right\vert ^{n} \geq
r-\dfrac{q\left( 1-\alpha \right) }{\left( q^{2}+1\right)
(k+1)-q(k+\alpha )}r^{2}.$$ The results follows by letting $r\to
1^-$.
\end{proof}

\begin{theorem}\label{th3.2}
Let $0\leq k<\infty,\ 0 < q < 1$ and $0\le \alpha <1$. Let the
function $f$ with the Taylor series $f(z) = z-a_2z^2-\cdots\ (a_n \ge 0)$ be a member of the class
$k$\mbox{-}$\widetilde{\mathcal{ST}}^-_q(\alpha )$. Then for $|z|= r
<1$
\begin{eqnarray}\label{eq13}
1-\dfrac{2q\left( 1-\alpha \right) }{\left( q^{2}+1\right) (k+1)-q(k+\alpha )%
}r &\leq &\ \left\vert f^{\prime }(z)\right\vert \ \leq
1+\dfrac{2q\left( 1-\alpha \right) }{\left( q^{2}+1\right)
(k+1)-q(k+\alpha )}r.
\end{eqnarray}
\end{theorem}

\begin{proof} Differentiating $f$ and using triangle inequality for the modulus, we have
\begin{equation}\label{eq14}
\left\vert f^{\prime }(z)\right\vert \leq 1+\overset{\infty }{\underset{n=2}{%
\sum }}na_{n}\left\vert z\right\vert ^{n-1}\leq 1+r\overset{\infty }{%
\underset{n=2}{\sum }}na_{n},
\end{equation}%
and%
\begin{equation}\label{eq15}
\left\vert f^{\prime }(z)\right\vert \geq 1-\overset{\infty }{\underset{n=2}{%
\sum }}na_{n}\left\vert z\right\vert ^{n-1}\geq 1-r\overset{\infty }{%
\underset{n=2}{\sum }}na_{n}.
\end{equation}%
The assertion \eqref{eq13}  now follows from \eqref{eq14} and
\eqref{eq15} by means of a rather simple consequence of \eqref{eq12}
given by%
\begin{equation*}
\overset{\infty }{\underset{n=2}{\sum }}na_{n}\leq \dfrac{2\left(
1-\alpha \right) }{\widetilde{\left[ 2\right] }_{q}(k+1)-(k+\alpha
)}.
\end{equation*}%
This completes the proof.
\end{proof}

\begin{theorem}\label{th4.1}
Let $0\leq k<\infty,\ 0 < q < 1$ and $0\le \alpha <1$, and set\vspace{-0.5em}
\[f_{1}(z)=z, \quad f_{n}(z)=z-\dfrac{1-\alpha }{\widetilde{\left[ n\right] }_{q}(k+1)-(k+\alpha )
}z^{n}\ \ (n=2,3,\ldots ).\]\vspace{-0.2em}
Then $f\in k$\mbox{-}$\widetilde{\mathcal{ST}}^-_q(\alpha )$ if, and only if, $f$ can be expressed in the form
\[
f(z)=\sum\limits_{n=1}^{\infty }\lambda _{n}f_{n}(z)\qquad (\lambda
_{n}>0,\ \sum\limits_{n=1}^{\infty }\lambda _{n}=1).\]
\end{theorem}
\begin{proof} Suppose that\vspace{-1.2em}
\begin{eqnarray*}
f(z) &=&\sum\limits_{n=1}^{\infty }\lambda _{n}f_{n}(z)=\lambda
_{1}f_{1}(z)+\sum\limits_{n=2}^{\infty }\lambda _{n}f_{n}(z) \\
&& \\
&=&\lambda _{1}f_{1}(z)+\sum\limits_{n=2}^{\infty }\lambda _{n}\left[ z-%
\dfrac{1-\alpha }{\widetilde{\left[ n\right] }_{q}(k+1)-(k+\alpha )}z^{n}\ \ %
\right] \\
&& \\
&=&\lambda _{1}z+\sum\limits_{n=2}^{\infty }\lambda
_{n}z-\sum\limits_{n=2}^{\infty }\lambda _{n}\dfrac{1-\alpha }{\widetilde{%
\left[ n\right] }_{q}(k+1)-(k+\alpha )}z^{n}\ \  \\
&& \\
&=&\left( \sum\limits_{n=1}^{\infty }\lambda _{n}\right)
z-\sum\limits_{n=2}^{\infty }\lambda _{n}\dfrac{1-\alpha
}{\widetilde{\left[
n\right] }_{q}(k+1)-(k+\alpha )}z^{n}\ \  \\
&& \\
&=&z-\sum\limits_{n=2}^{\infty }\lambda _{n}\dfrac{1-\alpha }{\widetilde{%
\left[ n\right] }_{q}(k+1)-(k+\alpha )}z^{n}\ .
\end{eqnarray*}%
Then\vspace{-0.8em}
\begin{eqnarray*}
\sum\limits_{n=2}^{\infty }\lambda _{n}\dfrac{1-\alpha }{\widetilde{\left[ n%
\right] }_{q}(k+1)-(k+\alpha )}\dfrac{\widetilde{\left[ n\right] }%
_{q}(k+1)-(k+\alpha )}{1-\alpha } &=&\sum\limits_{n=2}^{\infty
}\lambda _{n} =\sum\limits_{n=1}^{\infty }\lambda _{n}-\lambda _{1}
=1-\lambda _{1}\leq 1,
\end{eqnarray*}%
and we have $f\in k$\mbox{-}$\widetilde{\mathcal{ST}}^-_q(\alpha
)$.

Conversely, suppose that $f\in k$\mbox{-}$\widetilde{\mathcal{ST}}^-_q(\alpha )$. Since $| a_{n}| \leq (1-\alpha)/\big[\widetilde{\left[n\right] }_{q}(k+1)-(k+\alpha )\big]$, we may set\vspace{-.5em}
\begin{equation*}
\lambda _{n}=\dfrac{\widetilde{\left[ n\right] }_{q}(k+1)-(k+\alpha )}{%
1-\alpha }\left\vert a_{n}\right\vert \ \ \ \ and\ \ \ \lambda
_{1}=1-\sum\limits_{n=2}^{\infty }\lambda _{n}.
\end{equation*}%
Then\vspace{-.5em}
\begin{eqnarray*}
f(z) &=&z+\overset{\infty }{\underset{n=2}{\sum }}a_{n}z^{n}\\
&=&z+\sum%
\limits_{n=2}^{\infty }\lambda _{n}\dfrac{1-\alpha }{\widetilde{\left[ n%
\right] }_{q}(k+1)-(k+\alpha )}z^{n} \\
&=&z+\sum\limits_{n=2}^{\infty }\lambda
_{n}(z+f_{n}(z))=z+\sum\limits_{n=2}^{\infty }\lambda
_{n}z+\sum\limits_{n=2}^{\infty }\lambda _{n}f_{n}(z) \\
&=&\left( 1-\sum\limits_{n=2}^{\infty }\lambda _{n}\right)
z+\sum\limits_{n=2}^{\infty }\lambda _{n}f_{n}(z)\\
&=&\lambda
_{1}z+\sum\limits_{n=2}^{\infty }\lambda _{n}f_{n}(z) \\
&=&\sum\limits_{n=1}^{\infty }\lambda _{n}f_{n}(z),
\end{eqnarray*}%
and this completes the proof.
\end{proof}

\section{Hankel determinant}\setcounter{equation}{0}

Let $n$ and $s$ be the natural numbers, such that $n\geq 0$ and $s\geq 1$.
In 1976 Noonan and Thomas \cite{Noonan 76} defined the $s^{th}$
Hankel determinant of $f$ as
\begin{equation}\label{H}
H_{s}(n)=\left\vert
\begin{array}{llll}
a_{n} & a_{n+1} & \cdots & a_{n+s-1} \\
a_{n+1} & a_{n+2} & \cdots & a_{n+s} \\
\vdots & \vdots & \vdots & \vdots \\
a_{n+s-1} & a_{n+s} & \cdots & a_{n+2s-2}%
\end{array}%
\right\vert \ \ \ \ \ \ (a_{1}=1).
\end{equation}%
This determinant has been considered by several authors. For
example, Noor \cite{Noor 83} determined the rate of growth of $H_{s}(n)$ as $%
n\rightarrow \infty $ for functions $f$ given by (\ref{eq1}) with
bounded boundary. In particular, sharp upper bounds on $H_{2}(2)$, known as a second Hankel determinant,
were obtained in \cite{Noor 83, Owa 2010} for different classes of functions.

Note that%
\begin{equation*}
H_{2}(1)=\left\vert
\begin{array}{ll}
a_{1} & a_{2} \\
a_{2} & a_{3}%
\end{array}%
\right\vert =a_{3}-a_{2}^{2}, \quad \ H_{2}(2)=\left\vert
\begin{array}{ll}
a_{2} & a_{3} \\
a_{3} & a_{4}%
\end{array}%
\right\vert =a_{2}a_{4}-a_{3}^{2},
\end{equation*}%
and the first Hankel determinant $H_{2}(1)=a_{3}-a_{2}^{2}$ is  known as a special case of the
Fekete-Szeg\"{o} functional.

In this section will look more closely at the behaviour of the first and second Hankel determinant in the class $k$\mbox{-}$\widetilde{\mathcal{ST}}_q(\alpha )$,
additionally we find a bound of the Fekete-Szeg\"{o} functional and, as a special case, we obtain a bound of $|H_2(1)|$.
For convenience, in the sequel we use the abbreviations \begin{equation*}
q_{2}=\widetilde{\left[ 2\right] }_{q}-1,\quad q_{3}=\widetilde{\left[ 3%
\right]}_q-1,\quad q_{4}=\widetilde{\left[ 4\right] }_{q}-1, \quad {where}\quad 0< q <1.
\end{equation*}

\begin{theorem}
Let $0\leq k<\infty,\ 0 < q < 1$, $0\le \alpha <1$, and let $f\in
k$\mbox{-}$\widetilde{\mathcal{ST}}_q(\alpha )$.

\textbf{1. }If%
\begin{equation*}
U-P_1q_2(q_2q_4-1)\le 0,\quad V-P_1^2q_2^2q_4 \le 0,
\end{equation*}%
then the second Hankel determinant satisfies
\begin{equation*}
\left\vert a_{2}a_{4}-a_{3}^{2}\right\vert \leq \dfrac{P_1^2}{q_{3}^{2}}.
\end{equation*}

\textbf{2. }If%
$$
U-P_1q_2(q_2q_4-1)\ge 0,\quad 2S-U-P_1^2q_2(1+q_2q_4) \ge 0,$$ or
$$U-P_1q_2(q_2q_4-1)\le 0,\quad V-P_1^2q_2^2q_4 \ge 0,$$

then the second Hankel determinant satisfies%
\begin{equation*}
\left\vert a_{2}a_{4}-a_{3}^{2}\right\vert \leq \frac{V}{q_2^2q_3^2q_4}.
\end{equation*}
\textbf{3. }If%
$$U-P_1q_2(q_2q_4-1)> 0,\quad  2V-U-P_1^2q_2(1+q_2q_4) \le 0,$$
then
$$|a_{2}a_{4}-a_{3}^{2}|\le \frac{4\,P_1^2q_2^2q_4\,V-2P_1^2q_2(1+q_2q_4)\, U-U^2-P_1^4q_2^2(1+q_2q_4)^2}
{4\big(V-U-P_1^2q_2\big)q_2^2q_3^2q_4},$$
where $U,V$, and $M,N,S$ are given by
\begin{equation}\label{nms}\begin{array}{rcl}
U&=&|M+2P_1^2q_2+2P_1q_2q_4S|,\quad V=|M+N+P_1^2q_2-q_4S^2+2P_1q_2q_4S|,\\
N&=&P_1q_3\big[P_1^3+(P_3-2P_2)q_2q_3+P_1(P_2-P_1)(q_2+q_3)+P_1q_2q_3\big],\\
M&=&P_1q_3\big[2q_2q_3(P_2-P_1)+P_1^2(q_2+q_3)\big],\qquad S=P_1^2+q_2(P_2-P_1).\end{array}
\end{equation}
\end{theorem}

\begin{proof}
Let $f\in k$\mbox{-}$\widetilde{\mathcal{ST}}_q(\alpha )$. Then, there exists a Schwarz function $w,\ w(0)=1, |w(z)|<1$ for $z\in\mathbb{D}$, such that
\begin{equation*}
\dfrac{z(\widetilde{D}_{q}f)(z)}{f(z)}=p_{k,\alpha} (w(z)).
\end{equation*}%
Let
\begin{equation}\label{zd1}
p_0(z)=\dfrac{1+w(z)}{1-w(z)}=1+B_1z+B_2z^2+\cdots,
\end{equation}or, equivalently
$$ w(z)=\dfrac{p_{0}(z)-1}{p_{0}(z)+1}=\dfrac{1}{2}\left(B_{1}z+\left(B_{2}-\dfrac{B_1^2}{2}\right)z^{2}+\cdots\right).$$
Such function $p_0$ is analytic in the unit disk, and has a positive real part there. Using the Taylor expansion of $p_{k,\alpha}$ and $w$ we obtain
\begin{equation}\label{zd2}\begin{array}{rcl}
p_{k,\alpha}(w(z))&=& 1+\dfrac{P_1B_1}{2}z+\left(\dfrac{P_1 B_2}{2}+\dfrac{B_1^2(P_2-P_1)}{4}\right)z^2\\
&+&\left(\dfrac{P_1B_3+(P_2-P_1)B_1B_2}{2}+\dfrac{B_1^3(P_3+P_1)}{8}-\dfrac{P_2B_1^3}{4}\right)z^3+\cdots.\end{array}\end{equation}

Since
\begin{eqnarray*}
\dfrac{z(\widetilde{D}_{q}f)(z)}{f(z)} &=&1+q_2 a_{2}z+\left[q_3
a_3-q_2 a_2^2 \right] z^{2}+\left[ q_4
a_4-(q_2+q_3)a_2a_3+q_2a_{2}^{3}\right] z^{3}+\cdots,
\end{eqnarray*}
then, combining \eqref{zd1} with \eqref{zd2}, we have
\begin{equation}\label{a234}\begin{array}{rcl}
a_{2}&=&\dfrac{P_1B_{1}}{2q_{2}}, \quad a_{3}=\dfrac{1}{4q_{2}q_{3}}\left[P_1^2B_{1}^{2}-P_1B_1^2q_2+P_2B_1^2q_{2}+2P_1B_2q_2\right],\\
a_{4}&=&\dfrac{B_1^3 \big(P_1^3 + (P_3-2 P_2+P_1) q_2 q_3 + P_1(P_2-P_1) (q_2 + q_3)\big)
  }{8q_{2}q_{3}q_{4}}\\
 &+&\dfrac{2B_1B_2\big(P_1^2(q_2+q_3)+2q_2q_3(P_2-P_1)\big)+4B_3 P_1 q_2 q_3}{8q_{2}q_{3}q_{4}}.\end{array}\end{equation}

From the above we find that
$$H_2(2)=a_{2}a_{4}-a_{3}^{2}=\frac{B^4 N+(2B_2)B^2M+(4B_3)BP_1^2q_2q_3^2-\big[(2B_2)P_1q_2+B^2S\big]^2q_4}{16 q_2^2 q_3^2 q_4},$$
where, without loss of generality, we set $B:=B_1>0$, and  $N,M,S$ are given by \eqref{nms}.
Applying Lemma \ref{gr}, and performing the necessary computations we obtain
$$\qquad H_2(2)=\frac{B^4 \big[N+M+P_1^2q_2-q_4S^2+2P_1q_2q_4S\big]+xB^2(4-B^2)\big[M+2P_1^2q_2-2P_1q_2q_4S\big]}{16 q_2^2 q_3^2 q_4}$$

$$+\quad\frac{-x^2(4-B^2)\big[B^2P_1^2q_2+4P_1^2q_2^2q_4\big]+2B(4-B^2)(1-|x|^2)zP_1^2q_2q_3^2 }{16 q_2^2 q_3^2 q_4}.$$
Set now $\rho=|x|$, where $\le \rho\le 1$, and take an absolute value of $H_2(2)$. Applying additionally $|z|\le 1$, we have
$|H_2(2)|\le\Phi(\rho,B)=W(\alpha\rho^2+\beta\rho+\gamma)$, where
$$\alpha=(4-B^2)\big[B^2P_1^2q_2+4P_1^2q_2^2q_4\big]-2B(4-B^2)^2P_1^2q_2q_3^2,\ \beta=B^2(4-B^2)\big|M+2P_1^2q_2+2P_1q_2q_4S\big|,$$
$$\gamma = 2B(4-B^2)P_1^2q_2q_3^2+B^4 \big|N+M+P_1^2q_2-q_4S^2+2P_1q_2q_4S\big|,$$
and $W=1/(16q_2^2 q_3^2 q_4)$.
We note that $\alpha \ge 0, \beta \ge 0$. Indeed, an inequality $\beta \ge 0$ is obvious, and
we have $\alpha = (4-B^2)P_1^2q_2\big[B^2-2Bq_3^2+4q_2q_4\big]$. The expression in a square brackets
$\Psi(B)=B^2-2Bq_3^2+4q_2q_4$ is a quadratic function of $B, \ (0\le B\le 2)$ with roots at $B=2$, and $B=2(q_3^2-1)>2$. Since $\Psi(0)=4q_2q_4>0$ then $\Psi(B) > 0$ for $0\le B\le 2$.
Hence $\partial\Phi/\partial\rho =W(2\alpha\rho+\beta)\ge 0$, and from this fact we conclude that $\Phi$ is increasing function of $\rho$.
Therefore, for fixed $B\in [0,2]$, the maximum of $\Phi(\rho,B)$ is attained at $\rho =1$, that is $\max\Phi(\rho,B)=\Phi(1,B)=:G(B).$
We note that
$$\begin{array}{rcl}
G(B)=\dfrac{1}{16q_2^2 q_3^2 q_4}&\Bigg(&B^4\Big[\big|M+N+P_1^2q_2-q_4S^2+2P_1q_2q_4S\big|-\big|M+2P_1^2q_2+2P_1q_2q_4S\big|-P_1^2q_2\Big]\\
&+& B^2\Big[4\big|M+2P_1^2q_2+2P_1q_2q_4S\big|+4P_1^2q_2(1-q_2q_4)\Big]\\
&+&16P_1^2q_2^2q_4\Bigg).\end{array}$$
Let
\begin{equation}\label{pqr}\begin{array}{rcl}
P&=&\big|M+N+P_1^2q_2-q_4S^2+2P_1q_2q_4S\big|-\big|M+2P_1^2q_2+2P_1q_2q_4S\big|-P_1^2q_2,\\ \\

Q&=&4\big|M+2P_1^2q_2+2P_1q_2q_4S\big|+4P_1^2q_2(1-q_2q_4),\\ \\

R&=& 16P_1^2q_2^2q_4.\end{array}\end{equation}
Now, analyzing the maximum of a $Pt^2+Qt+R$, over $0\le t\le 4$, we conclude that
$$|H_2(2)|\le\dfrac{1}{16q_2^2 q_3^2 q_4}\left\{\begin{array}{lcl}
R&for&Q\le0,\ P\le -Q/4,\\ \\

16P+4Q+R&for&\big(Q\ge 0, P\ge -Q/8\big)\ \textit{or}\ \big(Q\le 0, P\ge -Q/4\big),\\ \\

R-Q^2/(4P)&for&Q>0,P\le -Q/8,\end{array}
\right.$$
where $P,Q,R$ are given by \eqref{pqr}. This completes the proof.
\end{proof}

\begin{corollary}
Let $q\to 1^-$. Then $k$\mbox{-}$\widetilde{\mathcal{ST}}_q(\alpha )\to k$\mbox{-}$\mathcal{ST}(\alpha)$, for which $P_{1}=\dfrac{8}{\pi ^{2}}$.
Then we get%
\begin{equation*}
|a_{2}a_{4}-a_{3}^{2}| \leq \dfrac{16}{\pi^2}.
\end{equation*}
\end{corollary}

\begin{theorem}Let $0\leq k<\infty,\ 0 < q < 1$, $0\le \alpha <1$, and let $f\in
k$\mbox{-}$\widetilde{\mathcal{ST}}_q(\alpha )$. Then for complex $\mu$ it holds
$$|a_3-\mu a_2^2|\le \dfrac{P_1^2|q_2-\mu q_3|+P_2q_2^2}{q_2^2 q_3}.$$  In the case, when $\mu$ is real, then
$$|a_3-\mu a_2^2|\le \left\{\begin{array}{rcl}
\dfrac{P_2 q^2}{q^4+1}+P_1^2q^2\dfrac{q(q^2-q+1)-\mu (q^4+1)}{(q^4+1)(q^2-q+1)^2}&for&\mu \le \dfrac{q(q^2-q+1)}{q^4+1},\\
\dfrac{P_2 q^2}{q^4+1}+P_1^2q^2\dfrac{\mu (q^4+1)-q(q^2-q+1)}{(q^4+1)(q^2-q+1)^2}&for&\mu \ge \dfrac{q(q^2-q+1)}{q^4+1}.\end{array}\right.$$
\end{theorem}\begin{proof}
We apply a form of $a_2, a_3$, given by \eqref{a234}, and assume as in the proof of the first part that $B:=B_1 >0$. Then, for complex $\mu$ we have

$$a_3-\mu a_2^2 = \frac{B^2\big(P_1^2q_2+q_2^2(P_2-P_1)-\mu P_1^2q_3\big)+(2B_2)P_1q_2^2}{4q_2^2q_3}.$$
Making use of Lemma \ref{gr}, we obtain
$$a_3-\mu a_2^2 = \frac{B^2\big(P_1^2q_2+q_2^2(P_2-P_1)-\mu P_1^2q_3\big)+(B^2+x(4-B^2))P_1q_2^2}{4q_2^2q_3},$$
where $x$ is a complex number satisfying $|x|\le 1$. Hence
$$a_3-\mu a_2^2 = \frac{B^2\big[q_2(P_1^2+P_2q_2)-\mu P_1^2q_3\big]+(4-B^2)P_1q_2^2}{4q_2^2q_3}.$$
After simplification and using $B\le 2$, we get

$$|a_3-\mu a_2^2| = \frac{\big|P_1^2(q_2-\mu q_3)+ P_2q_2^2\big|}{q_2^2q_3}.$$
We note also that $P_1, P_2$ are nonnegative, and $q_2, q_3$ are positive real number, therefore
$$|a_3-\mu a_2^2| =\frac{P_1^2\big|q_2-\mu q_3\big|+P_2q_2^2}{q_2^2q_3},$$
that establishes our first assertion.
For real $\mu$ our claim is deduced by the observation that $q_2= q+1/q-1$, and $q_3 = q^2+1/q^2$, where $0<q<1$.
\end{proof}
A trivial computation gives the bound for the first Hankel derivative, and for the third coefficient, below.
\begin{corollary}\label{c1} Let $0\leq k<\infty,\ 0 < q < 1$, $0\le \alpha <1$, and let $f\in
k$\mbox{-}$\widetilde{\mathcal{ST}}_q(\alpha )$. Then, the first Hankel determinant satisfy
$$|a_3- a_2^2|\le \dfrac{ q^2(P_2+P_1^2q)}{q^4+1}-\dfrac{P_1^2q^2}{q^2-q+1}.$$
\end{corollary}
\begin{corollary} Under the assumption the same as in the Corollary \ref{c1} we have
$$|a_3|\le \dfrac{q^2\big(P_2+P_1^2 q\big)}{q^4+1}.$$
\end{corollary}

\section{Acknowledgement}

The second author is supported by the Scientific and Technological Research Council of Turkey
(TUBITAK 2214A).


\begin{thebibliography}{99}

\bibitem{Bharati 97}
R. Bharati, R. Parvatham and A. Swaminathan, \textit{On subclasses of uniformaly convex functions and correspondding class of
starlike functions}, Tamkang J. Math. 28 (1997), 17-32.

\bibitem{Bredenharn 84}
L. C. Biedenharn, \textit{The quantum group $SU_q(2)$ and a $q$-analogue of the boson operators}, J. Phys., A 22 (1984),
L873-L878.

\bibitem{Brahim 2013}
K. L. Brahim and Y. Sidomou, \textit{On some symmetric $q$-special functions}, Le Matematiche, LXVIII (2013), 107-122.

\bibitem{Gasper 90}
G. Gasper and M. Rahman, 'Basic hypergeometric series', Cambridge Univ. Press, Cambridge, MA, 1990.

\bibitem{Goodman 91}
A. W. Goodman, \textit{On uniformly starlike functions}, J. Math. Anal. Appl. 155 (1991), 364-370.

\bibitem{Gr} U. Grenander and G. Szeg\"{o}, 'Toeplitz forms and their applications', California Monographs
in Mathematical Sciences Univ, California Press, Berkeley, 1958.

\bibitem{Jackson 08}
F. H. Jackson, \textit{On $q$-functions and a certain difference operator}, Transactions of the Royal Society of Edinburgh, 46 (1908), 253-281.

\bibitem{Kanas 2005}
S. Kanas, \textit{Coefficient estimates in subclasses of the Caratheodory class related to conical domains}, Acta Math. Univ. Comenian.
74(2)(2005), 149 – 161.

\bibitem{Kanas 2000a} S. Kanas and H.M. Srivastava, \textit{Linear operators associated with $k$-uniformly convex functions}, Integral Transforms Spec. Funct. 9 (2000), 121-132.

\bibitem{Kanas 2001}
S. Kanas and T. Yaguchi, \textit{Subclasses of $k$-uniformly convex and starlike functions defined by generalized derivative}, Publ. Inst. Math. (Beograd) (N.S),
 tome 69 (83) (2001), 91-100.

\bibitem{Kanas 2009}
S. Kanas, \textit{Norm of pre-Schwarzian derivative for the class of $k$-uniformly convex and $k$-starlike functions}, Appl. Math. Comput. 215 (2009), 2275-2282.

\bibitem{Kanas 2014}
S. Kanas and D. Raducanu, \textit{Some class of analytic functions related to conic domains}, Math. Slovaca, 64 (5) (2014), 1183--1196.

\bibitem{KanSu}
S. Kanas and T. Sugawa, \textit{Conformal representations of the interior of an ellipse}, Ann. Acad. Sci.Fenn. Math. 31(2006), 329--348.

\bibitem{Kanas 2000b}
S. Kanas and A. Wisniowska, \textit{Conic regions and k-uniformly starlike functions}, Rev. Roumaine Math. Pures Appl.  45 (4) (2000), 647--657

\bibitem{Noonan 76} J. W. Noonan and D. K. Thomas, \textit{On the second Hankel determinant of areally mean $p$-valent functions}, Trans. Amer. Math. Soc.
 223 (2) (1976), 337--346.

\bibitem{Noor 83} K. I. Noor,  \textit{Hankel determinant problem for the class of functions with bounded boundary rotation},  Rev. Roumaine Math. Pures Appl.
28 (c) (1983), 731--739.

\bibitem{Owa 2010} T. Hayami and S. Owa, \textit{Generalized Hankel determinant for certain classes}, Int. Journal of Math. Analysis 52 (4) (2010),
2473--2585.

\bibitem{Polatoglu 2016}
Y. Polato\u{g}lu, \textit{Growth and distortion theorems for generalized $q$-starlike functions}, Adv. Math., 5 (1) (2016), 7 12.

\bibitem{Pommerenke 75} C. Pommerenke, 'Univalent Functions', Vandenhoeck \& Ruprecht, G\"{o}ttingen. 1975.

\bibitem{Purohit 2011}
S. D. Purohit and R. K. Raina, \textit{Certain subclass of analytic functions associated with fractional $q$-calculus operators},
Math. Scand. 109 (2011), 55--70.

\bibitem{Ronning 93} F.  R\o nning, \textit{A survey on uniformly convex and uniformly starlike functions}, Ann. Univ. Mariae Curie-Skłodowska Sect. A, 47(13) (1993),
123-134.

\bibitem{Seker 2011}
B. \c{S}eker, M. Acu, and S. Sumer Eker, \textit{Subclasses of
\textit{k}-uniformly convex and $k$-starlike functions defined by Salagean operator}, Bull. Korean Math. Soc. 48 (1) (2011), 169-182.

\bibitem{Srivastava1989}
H. M. Srivastava, \textit{Univalent functions, fractional calculus, and associated generalized hypergeometric functions} in 'Univalent
Functions, Fractional Calculus and Their Applications' (H. M.Srivastava and S. Owa, Editors), Halsted Press (Ellis Horwood
Limited, Chichester), John Wiley and Sons, New York, Chichester, Brisbane and Toronto, 1989.

\bibitem{Srivastava 2000}
H.M. Srivastava and A.K. Mishra, \textit{Applications of fractional calculus to parabolic starlike and uniformly convex functions},
Comput. Math. Appl. 39 (3-4) (2000), 57--69.

\bibitem{Ucar 2016}
H. E. \"{O}zkan U\c{c}ar, \textit{Coefficient inequalties for $q$-starlike functions},  Appl. Math. Comp. 276 (2016), 122-126.
\end{thebibliography}
\end{document}